\documentclass[12 pt]{amsart}
\usepackage{amsmath,amssymb,amsbsy,amsfonts,amsthm,latexsym,amsopn,amstext,amsxtra,euscript,amscd,color,mathrsfs, tikz}

\PassOptionsToPackage{hyphens}{url}\usepackage{hyperref}

\usepackage[capbesideposition=outside,capbesidesep=quad]{floatrow}

\restylefloat{table}
\restylefloat{table}
            
\usepackage{multirow,caption}
            
\usepackage{amscd}
\usepackage{color,enumerate}

\newcommand{\RNum}[1]{\lowercase\expandafter{\romannumeral #1\relax}}

\usepackage[colorinlistoftodos,prependcaption,textsize=tiny]{todonotes}

\newtheorem{thm}{Theorem}[section]
\newtheorem{lem}[thm]{Lemma}

\newtheorem{conj}[thm]{Conjecture}

\newtheorem{thm-con}[thm]{Theorem-Conjecture}
\numberwithin{equation}{section}

\theoremstyle{definition}

\newcommand{\F}{\mathbb{F}} 
\def\Tr{{\rm Tr}}

\begin{document}
\title[A proof of a conjecture on permutation pentanomials]{A proof of a conjecture on permutation pentanomials}  
\author[K. Mallick]{Krishna Mallick}
\address{Department of Computer Science and Engineering, IIT Kharagpur, Kharagpur  721302, India}
\email{krishna.mallickmath@gmail.com}

\author[M. Pal]{Mohit Pal}
\address{Department of Informatics, University of Bergen, PB 7803, N-5020, Bergen, Norway}
\email{mathmohit@outlook.com}

\maketitle 
\begin{abstract}
    In this paper, we use finite fields and linear algebra methods to resolve a conjecture by T. Zhang, L. Zheng, H. Wang, J. Peng and Y. Li (Finite Fields Appl. 110 (2026) 102743) concerning permutation pentanomials.
\end{abstract}

\section{Introduction}

Let $\F_{q^n}$ be the finite field with $q^n$ elements, where $q$ is a prime power and $n$ is a positive integer. We denote by $\F_{q^n}^*$, the multiplicative cyclic group of non-zero elements of $\F_{q^n}$. It is well-known that every mapping from $\F_{q^n}$ into itself can be uniquely represented by a polynomial in $\F_{q^n}[X]$ of degree less than $q^n$. Accordingly, we may use the terms function and polynomial interchangeably. A polynomial $f\in \F_{q^n}[X]$ is called a permutation polynomial if the associated mapping $c \mapsto f(c)$ permutes the elements of $\F_{q^n}$. Permutation polynomials have been an active area of research since they have applications in various areas such as coding theory~\cite{DH13, YLC07}, cryptography~\cite{LM84, SH98} and combinatorial designs~\cite{DY06}. For more background material on permutation polynomials, we refer to Chap. 7 of~\cite{LN97}.

Permutation polynomials with a few terms are of particular interest due to their simple algebraic structures. The simplest class of polynomials is given by monomial permutation polynomials $X^e$, which permute $\F_{q^n}$ if and only if $\gcd(e, q^n-1) = 1$. Unlike monomials, the classification of permutation polynomials having a few terms, such as binomials, trinomials, quadrinomials, and pentanomials, is non-trivial and has not yet been completely resolved.  For a survey of recent developments on permutation binomials and trinomials over finite fields, the reader may refer to~\cite{Hou15BT}. We refer the interested reader to~\cite[Appendix]{GHLKP26} for recent advancements on permutation quadrinomials and pentanomials.

Recently, Zhang et al.~\cite{Zhang2026} studied permutation pentanomials of the form
\[
f(X)= \epsilon_0 X^{d_0} +L(\epsilon_1 X^{d_1}+\epsilon_2 X^{d_2}),
\]
where $L(X) = X^q+X$ and $d_0 \in \{ 1, 2, 4\}$. Based on the numerical results, the authors proposed the following conjecture.

\begin{conj}\cite[Conjecture 1]{Zhang2026} \label{C1}
    Let $m, k$ be positive integers with $q=2^m$. Then
    \[
    f_1(X)= X^{2^k+1}+X^{2^kq^2+1}+X^{q+2^k}+X^{q^2+2^k}+X^{2^k q+1}
    \]
    and
    \[
    f_2(X) = X^{2^k+1}+X^{(2^k+1)q}+X^{q+2^k}+ X^{q^2+2^k}+ X^{2^k q^2+q}
    \]
    are permutation polynomials over $\F_{q^3}$ if and only if $\gcd(2^k+1, q-1)=1$.
\end{conj}
It is the intent of our paper to completely prove this conjecture. The paper is organized as follows. In Section~\ref{S2}, we recall some basic notions and related results that will be used in the proof of the conjecture. In Section~\ref{S3}, we present a proof of Conjecture~\ref{C1}. Finally, we summarize the paper with some concluding remarks in Section~\ref{S4}.

\section{Preliminaries}\label{S2}

Throughout this paper, let $q=2^m$, where $m$ is a positive integer. Let $\Tr$ be the relative trace map from $\F_{q^3}$ to $\F_q$ defined as $\Tr(X):=X+X^q+X^{q^2}$. Let $C_0$ be a set defined as
\[
C_0 := \{X \in \F_{q^3} \mid \Tr(X)=0\}.
\]
It is well-known that $C_0$ is a $2$ dimensional vector space over $\F_q$. For any $u \in \F_q$, define the coset 
\[
C_u:= u+C_0= \{u+X \mid X \in C_0\}.
\]
We use the convention that $C_0$ is the coset corresponding to $u=0$. One may note that for any $Y \in C_u$, we have $Y= u+X$ for some $X \in C_0$ and hence 
\[
\Tr(Y) = \Tr(u+X) = \Tr(u)+\Tr(X)= u\Tr(1) =u.
\]
It is straightforward to see that $\lvert C_u \rvert= q^2$ for all $u \in \F_q$. Also, for any $u, v \in \F_q$ with $u \neq v$, we have $C_u \cap C_v = \phi$ as any $Y \in C_u \cap C_v$ would imply that $Y=u+X_1$ and $Y=v+X_2$ for some $X_1, X_2 \in C_0$. Therefore, we have
\[
 u+X_1 = v+X_2 \implies  \Tr(u+X_1) = \Tr(v+X_2) \implies u=v,
\]
which is a contradiction. Thus, we have the following partition of $\F_{q^3}$.
\[
\F_{q^3} = \bigcup_{u \in \F_q} C_u.
\]
In the next section, we shall use this partition of $\F_{q^3}$ to prove the permutation property of $f_1$ and $f_2$. We shall also use the following lemma in the proof of Conjecture~\ref{C1}. 
\begin{lem}~\cite[Lemma 2.1]{C99}~\label{L1}
   Let $\alpha, \beta$ be positive integers. Then
    \begin{equation*}
    \gcd(2^{\alpha}+1, 2^{\beta}-1) =
        \begin{cases}
         1~&~\mbox{if}~v_2(\beta) \leq v_2(\alpha),\\
         2^{\gcd(\alpha,\beta)}+1~&~\mbox{if}~v_2(\beta) > v_2(\alpha),\\
        \end{cases}
    \end{equation*}
where $v_2(\alpha), v_2(\beta)$ are the highest powers of $2$ that divide $\alpha, \beta$, respectively.
\end{lem}

\section{A proof of the conjecture}\label{S3}
In this section, we give a proof of Conjecture\ref{C1}. We prove the permutation property of $f_1$ and $f_2$, separately. The following theorem proves the first part of Conjecture~\ref{C1}.

\begin{thm}\label{T}
     Let $m, k$ be positive integers with $q=2^m$. Then
    \[
    f_1(X)= X^{2^k+1}+X^{2^kq^2+1}+X^{q+2^k}+X^{q^2+2^k}+X^{2^k q+1}
    \]
    is a permutation polynomial over $\F_{q^3}$ if and only if $\gcd(2^k+1, q-1)=1$.
\end{thm}
\begin{proof}
It is easy to observe that
\begin{equation*}
    \begin{split}
        f_1(X) &= X^{2^k+1} +X^{2^k}(X+X^q+X^{q^2}) +X(X^{2^k}+ X^{2^kq}+ X^{2^kq^2})   \\
        &= X^{2^k+1}+X^{2^k}\Tr(X) + X\left(\Tr(X)\right)^{2^k}\\
        &= X^{2^k+1}+X^{2^k}\Tr(X) + X\left(\Tr(X)\right)^{2^k} + \left(\Tr(X)\right)^{2^k+1} + \left(\Tr(X)\right)^{2^k+1} \\
        &= \left(X+\Tr(X)\right)^{2^k+1} + \left(\Tr(X)\right)^{2^k+1}.
    \end{split}
\end{equation*}
We may note that for any $Y \in C_u$, we have a unique $ X \in C_0$ such that $Y=u+X$. Thus, we have
\begin{equation*}
    \begin{split}
        f_1(Y) &= f_1(u+ X)\\
        &= \left(u+ X+\Tr(u+ X)\right)^{2^k+1} + \left(\Tr(u+ X)\right)^{2^k+1} \\
        &=  X^{2^k+1} + u^{2^k+1}.
    \end{split}
\end{equation*}
We first assume that $\gcd(2^k+1, q-1)=1$. Our aim is to show that $f_1$ is a permutation of $\F_{q^3}$. From Lemma~\ref{L1}, it is easy to see that
\begin{equation*}
    \begin{split}
        \gcd(2^k+1, q-1)=1  &\iff v_2(m) \leq v_2(k)\\
        &\iff v_2(3m) \leq v_2(k)\\
        &\iff \gcd(2^k+1, q^3-1)=1.
    \end{split}
\end{equation*}
Thus, $X^{2^k+1}$ is injective on $\F_{q^3}$. It is easy to observe that the restriction of the map $f_1$ on the coset $C_u$ is injective. As for any $Y_1, Y_2 \in C_u$, we have unique $X_1, X_2 \in C_0$ such that $Y_1=X_1+u$ and $Y_2=X_2+u$, moreover
\[
f_1(Y_1)=f_1(Y_2) \implies X_1^{2^k+1}+u^{2^k+1}=X_2^{2^k+1}+u^{2^k+1} \implies X_1=X_2 \implies Y_1=Y_2.
\]
Since $u$ is arbitrary, the map $f_1$ is injective on $C_u$ for all $u \in \F_q$. To prove the permutation property of $f_1$, it is sufficient to show that the sets $f_1(C_u):= \{f_1(Y) \mid Y \in C_u \}$ are disjoint. On the contrary, let us assume that $Z \in f_1(C_u) \cap f_1(C_v)$ for some $u, v \in \F_q$ and $u \neq v$. Then, we have $Z= X_1^{2^k+1} + u^{2^k+1}$ and $Z= X_2^{2^k+1} + v^{2^k+1}$  for some $ X_1, X_2 \in C_0$. Therefore, we have
\begin{equation}\label{TEq1}
    X_1^{2^k+1} + u^{2^k+1} = X_2^{2^k+1} + v^{2^k+1}
        \iff X_1^{2^k+1} + X_2^{2^k+1}=u^{2^k+1}+v^{2^k+1}.
\end{equation}
Now, since $\gcd(2^k+1, q-1)=1$, we have $u^{2^k+1} + v^{2^k+1} := w  \in \F_q^*$. We shall show that there does not exist any $ (X_1, X_2) \in C_0 \times C_0$ such that 
\begin{equation}\label{TEq2}
    X_1^{2^k+1} + X_2^{2^k+1} =w.
\end{equation}
Since $w\in \F_q^*$ neither of $X_1, X_2$ can be $0$. Now raising power $q$ to Equation~\eqref{TEq2} and adding it to Equation~\eqref{TEq2}, we have
\begin{equation}\label{TEq3}
    X_1^{2^k+1}+X_1^{q(2^k+1)} = X_2^{2^k+1} + X_2^{q(2^k+1)}.
\end{equation}
Now, we shall show that Equation~\eqref{TEq3} has only a trivial solution $X_1=X_2$, or in other words the map $g(X)= X^{2^k+1}+X^{q(2^k+1)}$ is injective on $C_0$. It is easy to see that $g(0)=0$. Next, we show that $g(X)\neq 0$ for all $X \in C_0 \backslash \{0\}$. Assume that $g(X)=0$ for some $X \in C_0 \backslash \{0\}$. Then we have
\begin{equation*}
\begin{split}
      & X^{2^k+1}(1+X^{(q-1)(2^k+1)})=0\\
      \implies &  X^{(q-1)(2^k+1)}=1\\
       \implies & X^{q-1}=1~\quad~\mbox{(as}~\gcd(2^k+1, q^3-1)=1)\\
        \implies & X\in \F_q^*,\\
\end{split}
\end{equation*}
which is a contradiction, as $C_0 \cap \F_q^*= \phi.$ Thus $g(X)\neq 0$ for all $X \in C_0 \backslash \{0\}$. Now, raising power $q-1$ to Equation~\eqref{TEq3}, we have
\begin{equation}\label{TEq4}
\begin{split}
     & (X_1^{2^k+1}+X_1^{q(2^k+1)})^{q-1} = (X_2^{2^k+1} + X_2^{q(2^k+1)})^{q-1}\\
     \implies & \frac{X_1^{q(2^k+1)}+X_1^{q^2(2^k+1)}}{X_1^{2^k+1} + X_1^{q(2^k+1)}} = \frac{X_2^{q(2^k+1)}+X_2^{q^2(2^k+1)}}{X_2^{2^k+1} + X_2^{q(2^k+1)}} \\
      \implies & \frac{X_1^{2^k+1}+X_1^{2^k+q}+X_1^{2^kq+1}}{X_1^{2^k+1} + X_1^{q(2^k+1)}} = \frac{X_2^{2^k+1}+X_2^{2^k+q}+X_2^{2^kq+1}}{X_2^{2^k+1} + X_2^{q(2^k+1)}}~\mbox{(as}~X_i^{q^2}=X_i^q+X_i, i \in \{1,2\}) \\
      \implies & \frac{X_1^{2^k+1}(1+X_1^{q-1}+X_1^{2^k(q-1)})}{X_1^{2^k+1}(1 + X_1^{(q-1)(2^k+1)})} = \frac{X_2^{2^k+1}(1+X_2^{q-1}+X_2^{2^k(q-1)})}{X_2^{2^k+1}(1 + X_2^{(q-1)(2^k+1)})}\\
       \implies & \frac{1+r+r^{2^k}}{1 + r^{2^k+1}} = \frac{1+s+s^{2^k}}{1 + s^{2^k+1}} \\
       \implies & r+r^{2^k}+s^{2^k+1}+rs^{2^k+1}+r^{2^k}s^{2^k+1}=s+s^{2^k}+r^{2^k+1}+r^{2^k+1}s+r^{2^k+1}s^{2^k} \\
        \implies & r+s+(r+s)^{2^k}+r^{2^k+1}+s^{2^k+1}+rs(r+s)^{2^k}+r^{2^k}s^{2^k}(r+s)=0,
\end{split}
\end{equation}
where $r=X_1^{q-1}$ and $s=X_2^{q-1}$. Now, we shall consider two cases, namely, $r+s=0$ and $r+s \neq 0$.

\textbf{Case 1.} Let $r+s=0$. In this case, we have $X_1^{q-1}=X_2^{q-1} \implies X_1= wX_2$ for some $w \in \F_q^*$. Putting the value of $X_2$ in Equation~\eqref{TEq3}, we have
\begin{equation}\label{TEq5}
    X_1^{2^k+1}+X_1^{q(2^k+1)} = w^{2^k+1} (X_1^{2^k+1}+X_1^{q(2^k+1)}) \implies  w^{2^k+1}=1.
\end{equation}
Since $\gcd(2^k+1, q^3-1)=1$, Equation~\eqref{TEq5} implies that $w=1$ and hence $X_1=X_2$. It is straightforward to see that $X_1=X_2$ cannot be a solution of Equation~\eqref{TEq2}.

\textbf{Case 2.} Let $r+s\neq 0$. In this case, dividing Equation~\eqref{TEq4} by $(r+s)^{2^k+1}$, we have
\begin{equation}\label{TEq6}
\begin{split}
     Z^{2^k}+Z+1=0.
\end{split}
\end{equation}
where $\displaystyle Z=\frac{1+r+rs}{r+s} \in \F_{q^3}$. Raising power $2^k$ to Equation~\eqref{TEq6} and adding it to Equation~\eqref{TEq6}, we have
\[
Z^{2^{2k}}+Z=0 \implies Z \in \F_{2^{\gcd(2k,3m)}}.
\]
Since $\gcd(2^k+1, q^3-1)=1$, we have $v_2(3m) \leq v_2(k)$ and hence $v_2(3m) < v_2(2k)$ and hence $\gcd(2k,3m)=\gcd(k,3m).$ But 
$Z \in \F_{2^{\gcd(k,3m)}} \implies Z^{2^k}+Z=0$, which is a contradiction to Equation~\eqref{TEq6}. This completes the proof.
\end{proof}

The following theorem proves the second part of Conjecture~\ref{C1}.
\begin{thm}\label{T1}
Let $m, k$ be positive integers with $q=2^m$. Then
\[
f_2(X) = X^{2^k+1} + X^{2^kq+q} + X^{q+2^k}+ X^{q^2+2^k}+ X^{2^kq^2+q}
\]
is a permutation polynomial over $\F_{q^3}$ if and only if $\gcd(2^k+1, q-1)=1$.
\end{thm}
\begin{proof}
It is easy to observe that
\begin{equation*}
    \begin{split}
        f_2(X) &= X^{2^k+q} +X^{2^k}(X+X^q+X^{q^2}) +X^q(X^{2^k}+ X^{2^kq}+ X^{2^kq^2})   \\
        &= X^{2^k+q}+X^{2^k}\Tr(X) + X^q\left(\Tr(X)\right)^{2^k}\\
        &= X^{2^k+q}+X^{2^k}\Tr(X) + X^q\left(\Tr(X)\right)^{2^k} + \left(\Tr(X)\right)^{2^k+q} + \left(\Tr(X)\right)^{2^k+q} \\
        &= \left(X+\Tr(X)\right)^{2^k+q} + \left(\Tr(X)\right)^{2^k+q}.
    \end{split}
\end{equation*}
It is easy to see that $f_2(X)$ is a permutation of $\F_{q^3}$ if and only if 
\begin{equation}
\begin{split}
    f_2'(X)&=(f_2(X))^{q^2}\\
    &=  \left(X+\Tr(X)\right)^{q^22^k+1} + \left(\Tr(X)\right)^{q^22^k+1} \\
     &=  \left(X+\Tr(X)\right)^{2^{k'}+1} + \left(\Tr(X)\right)^{2^k+1},
\end{split}
\end{equation}
where $k'=k+2m$, is a permutation of $\F_{q^3}$. We may note that for any $Y \in C_u$, we have a unique $ X \in C_0$ such that $Y=u+X$. Thus, we have
\begin{equation*}
    \begin{split}
        f_2'(Y) &= f_2'(u+ X)\\
        &= \left(u+ X+\Tr(u+ X)\right)^{2^{k'}+1} + \left(\Tr(u+ X)\right)^{2^k+1} \\
        &=  X^{2^{k'}+1} + u^{2^k+1}.
    \end{split}
\end{equation*}
We first assume that $\gcd(2^k+1, q-1)=1$. Our aim is to show that $f_2'$ is a permutation of $\F_{q^3}$. From Lemma~\ref{L1}, it is easy to see that
\begin{equation*}
    \begin{split}
        \gcd(2^k+1, q-1)=1  &\iff v_2(m) \leq v_2(k)\\
        &\iff v_2(3m) \leq v_2(k)\\
        &\iff v_2(3m) \leq v_2(k+2m)\\
        &\iff \gcd(2^{k'}+1, q^3-1)=1.
    \end{split}
\end{equation*}
Thus, $X^{2^{k'}+1}$ is injective on $\F_{q^3}$. Similarly to the previous theorem, the map $f_2'$ is injective on $C_u$ for all $u \in \F_q$. To prove the permutation property of $f_2'$, it is sufficient to show that the sets $f_2'(C_u):= \{f_2'(Y) \mid Y \in C_u \}$ are disjoint. On the contrary, let us assume that $Z \in f_2'(C_u) \cap f_2'(C_v)$ for some $u, v \in \F_q$ and $u \neq v$. Then, we have $Z= X_1^{2^{k'}+1} + u^{2^k+1}$ and $Z= X_2^{2^{k'}+1} + v^{2^k+1}$  for some $ X_1, X_2 \in C_0$. Therefore, we have
\begin{equation}\label{T1Eq1}
    X_1^{2^{k'}+1} + u^{2^k+1} = X_2^{2^{k'}+1} + v^{2^k+1}
        \iff X_1^{2^{k'}+1} + X_2^{2^{k'}+1}=u^{2^k+1}+v^{2^k+1}.
\end{equation}
Now, since $\gcd(2^k+1, q-1)=1$, we have $u^{2^k+1} + v^{2^k+1} := w  \in \F_q^*$. Similar to Theorem~\ref{T}, it can be shown that there does not exist any $(X_1, X_2) \in C_0 \times C_0$ such that 
\begin{equation}\label{T1Eq2}
    X_1^{2^{k'}+1} + X_2^{2^{k'}+1} =w.
\end{equation}
This completes the proof.
\end{proof}

Following the pattern of Theorem~\ref{T} and Theorem~\ref{T1}, it can be easily shown that the following pentanomial $f_3$  also induces a permutation of $\F_{q^3}$.

\begin{thm} \label{T2}
Let $m, k$ be positive integers with $q=2^m$. Then
\[
f_3(X) = X^{2^k+1} + X^{2^k+q} + X^{2^k+q^2}+ X^{q2^k+q^2}+ X^{q^22^k+q^2}
\]
is a permutation polynomial over $\F_{q^3}$ if and only if $\gcd(2^k+1, q-1)=1$.
\end{thm}

\section{Concluding Remarks}\label{S4}

In this paper, we proved a conjecture by Zhang et al.~\cite{Zhang2026} concerning permutation pentanomials. As an immediate consequence, we showed that the polynomials $f_1(X)$ and $f_2(X)$ in the conjecture belongs to a more general class of permutation polynomials of the form
\[
f(X)=\left(X+\Tr(X)\right)^{2^k+q^i} + \left(\Tr(X)\right)^{2^k+q^i},
\]
where $i \in \{0,1,2\}$ and $\gcd(2^k+1, q-1)=1$.

\end{document}